\documentclass[12pt]{article}

\usepackage{amsmath}
\usepackage{amsthm}
\usepackage{amsfonts}
\usepackage{mathrsfs}
\usepackage{stmaryrd}
\usepackage{setspace}
\usepackage{fullpage}
\usepackage{amssymb}
\usepackage{breqn}
\usepackage{enumitem}
\usepackage{bbold}
\usepackage{authblk}
\usepackage{comment}
\usepackage{hyperref}
\usepackage{pgf,tikz}
\usepackage{graphicx}
\usepackage{subcaption}

\newtheorem{thm}{Theorem}[section]%[chapter]

\newtheorem{lemma}[thm]{Lemma}

\newtheorem{proposition}[thm]{Proposition}
\newtheorem{prop}[thm]{Proposition}

\newtheorem{clm}[thm]{Claim}

\newcommand\ex{\ensuremath{\mathrm{ex}}}

\newcommand\cF{{\mathcal F}}

\newcommand\cH{{\mathcal H}}

\newtheorem*{thm*}{Theorem}
\newtheorem*{prop*}{Proposition}
\newcommand{\ignore}[1]{}

\title{Rainbow Turán problems for a matching and any other graph}
\author{
 \hspace{0.2em}
D\'{a}niel Gerbner\thanks{\small Alfr\'ed R\'enyi Institute of Mathematics. Email:
\small \texttt{gerbner.daniel@renyi.hu}.}\,, Shujing Miao\thanks{\small School of Mathematics and Statistics, and Hubei Key Lab--Math. Sci., Central China Normal University, Wuhan 430079, China. Email:
\small \texttt{sjmiao2020@sina.com}.}}

\date{}

\begin{document}

\maketitle

\begin{abstract}
For a family of graphs $\cF$, a graph is called $\cF$-free if it does not contain any member of $\cF$ as a subgraph. Given a collection of graphs $(G_1,\ldots,G_t)$ on the same vertex set $V$ of size $n$, a rainbow graph on $V$ is obtained by taking at most one edge from each $G_i$. We say that a collection is rainbow $\cF$-free if it contains no rainbow copy of any member of $\cF$. In this paper, we study the maximum values of $min_{i\in [t]}|E(G_i)|$, $\sum_{i=1}^{t}|E(G_i)|$ and $\prod_{i=1}^{t}|E(G_i)|$ among rainbow $\{F,M_{s+1}\}$-free collections $(G_1,\ldots,G_t)$ on $n$ vertices.
\end{abstract}



\section{Introduction}

A topic in extremal graph theory that has attracted a lot of researchers originated in the works of Tur\'an \cite{Tu}. It seeks to determine $\ex(n,\cF)$ for a given family $\cF$ of graphs, which is the largest number of edges in an $n$-vertex graph that does not contain any member of $\cF$ as a subgraph. If $\cF=\{F\}$, we use the notation $\ex(n,F)$ instead of $\ex(n,\cF)$. Tur\'an \cite{Tu} determined $\ex(n,K_k)$ for every $n$ and $k$. There are countless other results, see \cite{fursim} for a survey.

A particular line of research has recently been initiated by Alon and Frankl \cite{af}. They considered forbidding a graph $F$ together with a matching $M_{s+1}$. Their results were extended by the first author \cite{gerbner}. Let $\cF(F)$ denote the family of graphs we can obtain from $F$ by deleting an independent set.

We will consider a rainbow generalization of the Tur\'an problem. We are given a collection of graphs $(G_1,\dots,G_t)$ on the same vertex set $V$. We say that a subgraph of the union of these graphs is \textit{rainbow} if for each edge $uv$ of this subgraph we can choose an $i$ such that $uv\in E(G_i)$ and each $i$ is chosen at most once. We say that a collection is \textit{rainbow $\cF$-free} if there is no rainbow copy of any member of $\cF$. 

Our goal is still to find a collection with many edges, but now there is no unique way to measure the density of our collection. We consider the following three functions.

\begin{itemize}
    \item $\ex_t(n,\cF)$ denotes the largest integer $e$ such that there is a rainbow $\cF$-free collection $(G_1,\dots, G_t)$ on $n$ vertices, such that each $G_i$ has at least $e$ edges.

    \item $\ex_t^\sum(n,\cF)$ is the largest $\sum_{i=1}^t |E(G_i)|$ among rainbow $\cF$-free collections $(G_1,\dots, G_t)$ on $n$ vertices.

        \item $\ex_t^\prod(n,\cF)$ is the largest $\prod_{i=1}^t |E(G_i)|$ among rainbow $\cF$-free collections $(G_1,\dots, G_t)$ on $n$ vertices.
\end{itemize}

In the case $\cF=\{F\}$, again we will use $F$ instead of $\{F\}$ inside the brackets.

The study of $\ex_t(n,\cF)$ was initiated by Aharoni, DeVos, de la Maza, Montejano and Šámal \cite{admms} for triangles. They determined $\ex_3(n,K_3)$ asymptotically. The bound is slightly larger than the trivial lower bound $\ex(n,K_3)$.

The study of $\ex_t^\sum(n,\cF)$ was initiated by Keevash, Saks, Sudakov and Verstra\"ete \cite{kssv}, who considered two trivial lower bounds: $\ex_t^\sum(n,\cF)\ge (|E(F)|-1)\binom{n}{2}$ because we can take complete graphs in $|E(F)|-1$ colors, $\ex_t^\sum(n,\cF)\ge t\ex(n,F)$ because we can take an $n$-vertex $F$-free graph with $\ex(n,F)$ edges in each of the colors. Among many other results, they showed that one of these two bounds is sharp for $\cF=\{K_k\}$ and $n$ sufficiently large.

The study of $\ex_t^\prod(n,\cF)$ was initiated by Frankl \cite{frankl} and continued in \cite{rainb}. Note that these are the only papers we are aware of that deal with $\ex_t^\prod(n,\cF)$. We remark that $\ex_t^\prod(n,\cF)$ has been studied only for $K_3$ and the 4-vertex path $P_4$, with some remarks concerning larger cliques and longer paths in \cite{rainb}. 
%Let us mention that $\ex_t^\prod(n,F)$ has been studied much less than the other two functions. We are aware only if some results and preliminary observations in the cases $F$ is a clique or a path in \cite{rainb}.

More generally, the triples $a_1,a_2,a_3$ such that a collection with $|E(G_i)|\ge a_i$ must contain a rainbow triangle were determined in \cite{fmr} for $n$ sufficiently large. In particular, this determines $\ex_t^\prod(n,K_3)$. See the survey \cite{sww} for more on these and other rainbow problems in graphs and hypergraphs.

Let us note that the expression ``rainbow Turán'' is also used in the literature for another notion \cite{kmsv}, where we consider proper colorings of a graph without a rainbow copy of $F$. In our language, it means that each $G_i$ is a matching, $G_i$ and $G_j$ are edge-disjoint, we do not have a bound on the number of colors, and we want to determine the largest sum of the number of the edges.

A reason to consider forbidding $M_{s+1}$ together with another graph is that there is a large amount of work on rainbow matchings, see e.g., \cite{abchs} and the references in it.

We will always assume $t\ge s+1$ and $t\ge |E(F)|$. Let $F$ be a graph and $p$ be an integer. We define $\mathcal{F}[p]=\{K_{p+1}\}$ if $F$ has no covering of size at most $p$; otherwise $\mathcal{F}[p]=\{F[S] :$ S is a covering of F with $|S|\le p\}$. Let $\cF(F)$ denote the family of graphs we obtain from $F$ by deleting an independent set. Let $F$ be a bipartite graph and let $p(F)$ denote the smallest possible order of a color class in a proper two‐coloring of $F$.

The Erd\H os-S\'os conjecture \cite{erd} states that for any $t$-vertex tree $T$, we have $\ex(n,T)\le (t-2)n/2$. It is known to hold for several classes of graphs, including paths \cite{eg}. We prove analogous versions for $\ex_t(n,\{F,M_{s+1}\})$ of multiple results on $\ex(n,\{F,M_{s+1}\})$ from \cite{gerbner}.

\begin{thm}\label{min} Let $t\ge\max\{|E(F)|,s+1\}$.

    \textbf{(i)} If $F$ is not bipartite and $n$ is sufficiently large, then $\ex_t(n,\{F,M_{s+1}\})=s(n-s)+\ex_t(s,\cF(F))$.

    \textbf{(ii)} If $F$ is bipartite and $p(F)>s$, then $\ex_t(n,\{F,M_{s+1}\})=s(n-s)+ex_t(s,\mathcal{F}[s])$.

    \textbf{(iii)} If $F$ is bipartite, $p(F)\le s$ and $t$ is sufficiently large, then $\ex_t(n,\{F,M_{s+1}\})=(p-1)n+O(1)$.

    \textbf{(iv)} Assume that $F$ is a balanced tree, i.e., $|V(F)|=2p(F)$, 
     and the Erd\H os-S\'os conjecture holds for $F$. If $p(F)\le s$ and $t$ is sufficiently large, then for sufficiently large $n$, we have $\ex_t(n,\{F,M_{s+1}\})=(p-1)(n-p+1)+ex_t(p-1,\mathcal{F}[p-1])$.
\end{thm}

Let us remark that \textbf{(iii)} does not always hold for smaller $t$. Consider 
%for simplicity 
$F=K_{p,p}$ and let $A$ be a $(p-1)$-element set, $B$ be an $(s-p+1)$-element set and $C$ be an $(n-s)$-element set. Add all the edges between $A$ and $C$ in every color, and for each vertex $v\in B$, we add all the edges between $v$ and $C$ in $p-1$ colors (to be specified later). The union of the resulting graphs does not contain $M_{s+1}$, and each copy of $K_{p,p}$ contains a vertex of $B$. But those vertices are incident to only $p-1$ colors, while each vertex of a rainbow copy of $K_{p,p}$ is incident to $p$ colors. Thus this construction is rainbow $\{K_{p,p}, M_{s+1}\}$-free. Now we specify the colors picked for vertices $v\in B$. For each edge we assign the $p-1$ colors such that the $t$ colors are divided as evenly as possible. Then each color is assigned to at least $\lfloor (s-p+1)(p-1)/t\rfloor$ vertices, thus altogether there are at least $(p-1+\lfloor (s-p+1)(p-1)/t\rfloor)(n-s)$ edges of each color.

\begin{thm}\label{sum} Let $t\ge\max\{|E(F)|,s+1\}$.

  \textbf{(i)} $\ex_t^\sum(n,\{F,M_{s+1}\})=\ex_s^\sum(n,F)+O(n)$.

  \textbf{(ii)} For $n$ sufficiently large, we have $\ex_t^\sum(n,\{K_3,M_{s+1}\})=\ex_s^\sum(n,K_3)$.
\end{thm}

Note that $\ex_s^\sum(n,F)$ is obviously quadratic if $F$ has at least two edges, since we can take $K_n$ in one color and the empty graph in the other colors.

Also note that we know $\ex_s^\sum(n,K_3)$ exactly: obviously $\ex_s^\sum(n,K_3)=s\binom{n}{2}$ for $s\le 2$, while $\ex_3^\sum(n,K_3)=n(n-1)$ and $\ex_s^\sum(n,K_3)=s\lfloor n^2/4\rfloor$ for $s\ge 4$. This follows from the result of Keevash, Saks, Sudakov and Verstra\"ete \cite{kssv} mentioned earlier.

Let us turn to $\ex_t^\prod(n,\{F,M_{s+1}\})$. A particular reason to forbid a matching is that we can easily determine the order of magnitude in other cases. If we forbid any graph $F$ that contains a vertex of degree larger than 1, then $t$ vertex-disjoint monochromatic copies of $K_{\lfloor n/t\rfloor}$ of different colors show that $\ex_t^\prod(n,F)=\Theta(n^{2t})$.

We determine the order of magnitude of $\ex_t^\prod(n,\{F,M_{s+1}\})$ for each $F$.
Let $S_r$ denote the star with $r$ leaves.

\begin{thm}\label{prod} Let $t\ge\max\{|E(F)|,s+1\}$.

\textbf{(i)} If $n$ is sufficiently large, then $\ex_t^\prod(n,M_{s+1})=(n-1)^{t-s+1}\binom{n}{2}^{s-1}$.

  \textbf{(ii)} If $F$ is not a star with isolated edges added, then $\ex_t^\prod(n,\{F,M_{s+1}\})=\Theta(n^{t+s-1})$.

  \textbf{(iii)}
    \begin{displaymath}
\ex_t^\prod(n,\{S_r,M_{s+1}\})=
\left\{ \begin{array}{l l}
\Theta(n^{2s-2}) & \textrm{if\/ $r=2$},\\
\Theta(n^{s(r-1)-1}) & \textrm{if\/ $t>s(r-1)$ and $r>2$},\\
\Theta(n^{s(r-1)}) & \textrm{if\/ $t=s(r-1)$},\\
\Theta(n^{t+s-\lceil (t-s)/(r-2)\rceil}) & \textrm{otherwise}.\\
\end{array}
\right.
\end{displaymath}

  \textbf{(iv)}  If $F$ is a star $S_r$ with $1\le m\le s-1$ isolated edges added, then 
  \begin{displaymath}
\ex_t^\prod(n,\{F,M_{s+1}\})=
\left\{ \begin{array}{l l}
\Theta(n^{2s-2}) & \textrm{if\/ $r=2$},\\
\Theta(n^{t+s-1}) & \textrm{if\/ $r-1\ge t-s+1$},\\
\Theta(n^{t+m-1}) & \textrm{if\/ $t\ge s(r-1)$},\\
\Theta(n^{t+m-1}) & \textrm{if\/ $r+s-2<t<s(r-1)$ and $m>\frac{s(r-1)-t}{r-2}$}.\\
\Theta(n^{t+\lfloor \frac{s(r-1)-t}{r-2}\rfloor})  & \textrm{if\/ $r+s-2<t<s(r-1)$ and $m\le\frac{s(r-1)-t}{r-2}$}.\\
\end{array}
\right.
\end{displaymath}
 
 % \begin{displaymath} 
 %\ex_t^\prod(n,\{S_r,M_{s+1}\})=
%\left\{ \begin{array}{l l}
%(n-1)^{t-s+1}\binom{n-1}{2}^{s-1} & \textrm{if\/ $r-1\ge t-s+1$},\\
%(n-1)^{r-2}\binom{n}{2}^{s-1} & \textrm{otherwise}.\\
%\end{array}
%\right.
%\end{displaymath}
\end{thm}

Note that if by deleting any vertex of $F$ we obtain a graph with at least $s$ edges, then $\ex_t^\prod(n,\{F,M_{s+1}\})=(n-1)^{t-s+1}\binom{n}{2}^{s-1}$, as shown by the upper bound in \textbf{(i)} and the construction of taking a complete graph $K_n$ in $s-1$ colors and star $S_{n-1}$ on the same vertex set in each of the other colors.

\section{Preliminaries}

Let us first consider the case a matching is forbidden without any other graph. Meshulam, see \cite{ah} showed the following.

\begin{proposition}[Meshulam]\label{meshu}
    $\ex_t(n,M_{s+1})=s(n-s)+\binom{s}{2}$.
\end{proposition}

One can often use a greedy algorithm to find rainbow subgraphs. The basic idea is that if an edge $uv$ in a subgraph $H$ has at least $|E(H)|$ colors, and the rest of $H$ is rainbow, there are at most $|E(H)|-1$ colors are used on the rest of the edges, thus one of the colors of $uv$ makes the whole $H$ rainbow. In particular, if every edge of $H$ is in at least $|E(H)|$ colors, then we can go through the edges of $H$ greedily and pick an available color for each edge to obtain a rainbow $H$. 
Here we formulate two simple statements concerning greedily found rainbow matchings.

\begin{lemma}\label{greed}
\textbf{(i)} Let $M_0$ be a rainbow matching of size $p$ in the collection $(G_1,\dots, G_p)$. For $p<i\le q$, let $v_i$ be distinct vertices not in any edges of $M_0$ such that $v_i$ has degree at least $2q-1$ in $G_i$. Then there is a rainbow matching of size $q$ in the collection $(G_1,\dots, G_q)$.

%\textbf{(i)} Let $v_i$ be distinct vertices such that $v_i$ has degree at least $2q-1$ in $G_i$. Then there is a rainbow matching of size $q$ in the collection $(G_1,\dots, G_q)$.

\textbf{(ii)}    If for each $i\le q$, $G_i$ contains at least $i$ vertices of degree at least $2q-1$, then there is a rainbow matching of size $q$ in the collection $(G_1,\dots, G_q)$.
\end{lemma}

\begin{proof} To prove \textbf{(i)}, for each $i$ we pick one of the neighbors $u_i$ of $v_i$ such that $u_i$ is not among the vertices $v_j$, not in any edges of $M_0$, and not among the vertices already picked. We can pick $u_i$ since we have to avoid at most $2q-2$ other vertices.

To prove \textbf{(ii)}, first we pick distinct vertices $v_i$ with degree at least $2q-1$ in $G_i$, for each $i$. We do this by going through the graphs $G_i$ in increasing order. We can always pick a new vertex, since there are at most $i-1$ vertices picked earlier, and we have at least $i$ choices. Then we apply \textbf{(i)} to complete the proof.
\end{proof}

Given a collection $(G_1,\dots, G_t)$ and an integer $s$, we say that a color $i$ is \textit{strong} if any rainbow matching of size $s'\le s$ avoiding color $i$ can be extended to a rainbow matching of size $s'+1$ with an edge of color $i$.

\begin{lemma}\label{strong}
    \textbf{(i)} If $G_i$ has more than $2s(n-2s)+\binom{2s}{2}$ edges, then $i$ is a strong color.

    \textbf{(ii)} If $G_i$ has at least $s(n-s)$ edges and less than $s$ vertices of degree at least $n/2s$, then $i$ is a strong color.

    %\textbf{(iii)} If $G_i$ has at least $2s+1$ vertices of degree at least $2s+1$, then $i$ is a strong color. kell ez???

    \textbf{(iii)} If $G_i$ contains $M_{2s+1}$, then $i$ is a strong color.
\end{lemma}

\begin{proof}
    Let us consider a rainbow matching $M$ of size $s'$ avoiding color $i$. Then there are at most $2s'$ vertices in $M$, and there are at most $2s'(n-2s')+\binom{2s'}{2}$ edges incident to those vertices. If $G_i$ has more than that many edges, one of them extends $M$ to a larger matching, proving \textbf{(i)}. If $G_i$ contains $M_{2s+1}$, then at least one of the edges of this $M_{2s+1}$ avoids each vertex of $M$, thus extends $M$ to a larger matching, proving \textbf{(iii)}.

    Let $U$ denote the set of vertices of degree at least $n/2s$ in $G_i$, and let $V$ denote the set of other vertices. Since there are at most $(s-1)n$ edges incident to $U$, there are at least $n$ edges inside $V$ in $G_i$. At most $2s$ vertices belong to $M$. Those vertices belonging to $M$ that are in $V$ are incident to less than $n$ edges of color $i$, thus there are other edges of color $i$ inside $V$. Any of them extends $M$ to a larger matching, proving \textbf{(ii)}.

\end{proof}

Keevash, Saks, Sudakov and Verstra\"ete showed (Lemma 2.1 in  \cite{kssv}) that if we are interested in $\ex_t^\sum(n,F)$, we can assume that $G_1\supseteq G_2\supseteq \dots \supseteq G_t$. The proof trivially extends to multiple forbidden subgraphs, we state it in this form below.

\begin{lemma}\label{nestlemma}
    Let $(G_1,\dots,G_t)$ be a rainbow $\cH$-free collection of graphs. Then there exists a rainbow $\cH$-free collection of graphs $(G_1',\dots,G_t')$ satisfying $G_1'\supseteq G_2'\supseteq \dots \supseteq G_t'$, such that for each pair of vertices $u,v$, the same number of graphs contain the edge $uv$ in $(G_1',\dots,G_t')$ as in $(G_1,\dots,G_t)$.
\end{lemma}

For the sake of completeness, we present the proof below.

\begin{proof}
If $G_i=G_j$ for all $i,j$, then we are done. Suppose there exists $i,j$ such that $G_i\neq G_j$. Consider a collection of graphs $(G_1',\dots,G_t')$ with $G_k'=G_k$ for $k\neq i,j$ and $G_i'=G_i\cap G_j$, $G_j'=G_i\cup G_j$. Then for each pair of vertices $u,v$, the same number of graphs contains the edge $uv$ in $(G_1',\dots,G_t')$ as in $(G_1,\dots,G_t)$. Suppose the collection of graphs $(G_1',\dots,G_t')$ contains a rainbow copy of a graph $H\in\cH$. Then this copy of rainbow $H$ must contain an edge $e\in E(G_i')$ and an edge $e'\in E(G_j')$, since the collection $(G_1,\dots,G_t)$ is rainbow $\cH$-free. We may assume $e'\in E(G_j)$. Note that $e\in E(G_i)$. Then we can color $e$ with color $i$ and $e'$ with color $j$ in $(G_1,\dots,G_t)$. So we find a copy of rainbow $\cH$ in $(G_1,\dots,G_t)$, a contradiction. Continuing the process we obtain a collection $(G_1',\dots,G_t')$ satisfying $G_1'\supseteq G_2'\supseteq \dots \supseteq G_t'$ from $(G_1,\dots,G_t)$ by a finite number of steps. This completes the proof.
\end{proof}

We deal separately with the case of forbidden $M_2$.

\begin{proposition}\label{m2}
    If $(G_1,\dots,G_t)$ is a rainbow $M_2$-free collection of graphs on at least 4 vertices, then either there is a vertex $v$ such that each edge in each color is incident to $v$, or all but at most one of the graphs have at most 4 edges.
\end{proposition}

\begin{proof}
    Assume indirectly that there are two independent edges $uv$ and $u'v'$ that each belong to some graphs in the collection. Clearly, they belong to the same graph, say $G_1$. Each edge in another color consists of one of the vertices $u,v$ and one of the vertices $u',v'$, thus there are 4 choices.
\end{proof}

We will use a simple statement about rainbow stars. Note that it would not be difficult to strengthen this, but it is sufficient for our purposes.

\begin{proposition}\label{starprop}
    Let $(G_1,\dots,G_t)$ be a collection of graphs and $v$ be a vertex. If $v$ is not the center of a rainbow $S_p$, then there is a set of at most $p-1$ edges incident to $v$ such that all but $p-1$ colors appear only on those $p-1$ edges among the edges incident to $v$.
\end{proposition}

\begin{proof}
   Let us consider an auxiliary bipartite graph $H$. Part $A$ consists of the edges incident to $v$ and part $B$ consists of the colors on those edges. The edge $uv$ is joined to the color $i$ if $uv$ has color $i$. A matching of size $p$ in $H$ would correspond to a rainbow $S_p$, thus there is no such matching. If $|A|\ge p$, then by Hall's condition there is a subset $A'$ of $A$ with neighborhood $B'$ smaller than $|A'|$. We pick such a set with the smallest $|B'|$ and delete both $A'$ and $B'$. Afterwards, we repeat this as long as we can. At the end, there is no set violating Hall's condition, thus there are fewer than $p$ vertices remaining in $A$. 
   
   We claim that when we delete $A'$ and $B'$, we delete a matching of size $|B'|$. Indeed, let $H'$ denote $H$ restricted to $A'\cup B'$. If there is no matching covering $B'$ in $H$, then by Halls' condition there is $B''\subset B'$ with neighborhood $A''$ in $H'$ such that $|A''|<|B''|$. Then the neighborhood of $A'\setminus A''$ in $H$ is a subset of $B'\setminus B''$, in particular it has size smaller than $|A'\setminus A''|$. But then we would have deleted $A'\setminus A''$ and its neighborhood instead of $A'$ and $B'$, a contradiction. This implies that we have deleted a matching of size $B'$. We repeat this, and altogether we delete a matching of size at most $p-1$, thus at most $p-1$ colors. This completes the proof.
\end{proof}

%If a vertex $v$ is not the center of a rainbow $S_p$, then there is a set of at most $p-1$ edges such that all but $p-1$ colors appear only on those $p-1$ edges. Let us consider an auxiliary bipartite graph $H_1$. Part $A_1$ consists of the edges incident to $v$ and part $B_1$ consists of the colors on those edges. The edge $uv$ is joined to the color $i$ if $uv$ has color $i$. A matching of size $p$ in $H_1$ would correspond to a rainbow $S_p$, thus there is no such matching. If $|A_1|\ge p$, then by Hall's condition there is a subset $A_1'$ of $A_1$ with neighborhood smaller than $|A_1'|$. We pick such a set with the smallest neighborhood $A_1'$ and delete both $A_1$ and $A_1'$. Afterwards, we repeat this as long as we can. At the end, there is no set violating Hall's condition, thus there are fewer than $p$ vertices remaining in $A_1$... maybe move this out as a lemma...

\section{Proofs of Theorems}

Let us continue with the proof of Theorem \ref{min} that we restate here for convenience.

\begin{thm*} Let $t\ge\max\{|E(F)|,s+1\}$.

    \textbf{(i)} If $F$ is not bipartite and $n$ is sufficiently large, then $\ex_t(n,\{F,M_{s+1}\})=s(n-s)+\ex_t(s,\cF(F))$.

    \textbf{(ii)} If $F$ is bipartite and $p(F)>s$, then $\ex_t(n,\{F,M_{s+1}\})=s(n-s)+ex_t(s,\mathcal{F}[s])$.

    \textbf{(iii)} If $F$ is bipartite, $p(F)\le s$ and $t$ is sufficiently large, then $\ex_t(n,\{F,M_{s+1}\})=(p-1)n+O(1)$.

 \textbf{(iv)} Assume that $F$ is a balanced tree, i.e., $|V(F)|=2p(F)$, 
     and the Erd\H os-S\'os conjecture holds for $F$. If $p(F)\le s$ and $t$ is sufficiently large, then for sufficiently large $n$, we have $\ex_t(n,\{F,M_{s+1}\})=(p-1)(n-p+1)+ex_t(p-1,\mathcal{F}[p-1])$.     
\end{thm*}

\begin{proof} For the lower bound of \textbf{(i)}, we take a copy of $K_{s,n-s}$ in each color, and a rainbow $\cF(F)$-free collection inside the part of order $s$.

    For the upper bound of \textbf{(i)}, consider a rainbow $\{F,M_{s+1}\}$-free collection $(G_1,\dots,G_t)$ such that each $G_i$ has at least $s(n-s)+\ex_t(s,\cF(F))$ edges. First, we claim that there is no strong color. Indeed, the other at least $s$ colors contain a rainbow $M_s$ by Proposition \ref{meshu}, and the strong color would extend that to an $M_{s+1}$. Using Lemma \ref{strong} \textbf{(ii)}, we obtain that in each color there are at least $s$ vertices of degree at least $n/2s$.

    Let us define an auxiliary bipartite graph $H$. Part $A$ consists of the $t$ colors, part $B$ consists of the vertices, and color $i$ is joined to vertex $v$ if $v$ has degree at least $n/2s$ in color $i$. By Lemma~\ref{greed}, a matching of size $s+1$ in $H$ gives a rainbow $M_{s+1}$ in the collection $(G_1,\dots,G_t)$. 

    If there is no matching of size $s+1$ in $H$, then any $(s+1)$-set $A'$ in $A$ is not covered by a matching, thus by Hall's condition there is a set $A''\subseteq A'$ of size at most $s+1$ such that its vertices altogether have less than $|A''|$ neighbors in $B$. Since each vertex of $A$ has at least $s$ neighbors in $B$, this implies that the vertices of $A''$ have the same $s$ neighbors in $B$ and $|A''|=s+1$, i.e., $A''=A'$. This holds for every $(s+1)$-subset of $A$, thus each vertex of $A$ is incident to the same $s$ vertices of $B$.

    This means that there is a set $S$ of $s$ vertices 
    %$v_1,\dots,v_s$ 
    that have degree at least $n/2s$ in each color. We claim that every edge contains at least one of the vertices of $S$. Indeed, if $uv$ is an edge avoiding $S$, we consider it as a rainbow matching of size $1$, take $s$ other colors and apply \textbf{(i)} of Lemma~\ref{greed} to find a rainbow matching of size $s+1$. We obtain that in each $G_i$, we have at most $\binom{s}{2}$ edges inside $S$ and at most $s(n-s)$ edges between $S$ and the rest of the vertices. If any color has at least $\binom{s}{2}$ edges missing between $S$ and the rest of the vertices, we are done. This means that at least $n-s-t\binom{s}{2}$ vertices are joined to each vertex of $S$ in each color. If there is a rainbow graph from $\cF(F)$ inside $S$, then it is extended to a rainbow $F$ by picking the remaining vertices among those at least $n-s-t\binom{s}{2}$ vertices. Therefore, the edges of the graphs $G_i$ inside $S$ must form a rainbow $\cF(F)$-free collection, thus for at least one of the graphs $G_i$ there are at most $\ex_t(s,\cF(F))$ edges inside $S$. Clearly, there are at most $s(n-s)$ other edges in color $i$, completing the proof of \textbf{(i)}. 
    
    For the lower bound of \textbf{(ii)}, we take a copy of $K_{s,n-s}$ in each color, and a rainbow $\mathcal{F}[s]$-free collection inside the part of order $s$. The resulting collection is rainbow $M_{s+1}$, since $s$ vertices are incident to all edges and rainbow $F$-free by the definition of $\mathcal{F}[s]$. 
    
    For the upper bound of \textbf{(ii)}, we can first claim that there exists a set $S$ of s vertices such that every edge contains at least one vertex of $S$ in each color. This assertion can be obtained through the proof process of \textbf{(i)}. Then we obtain that in each $G_i$, we have at most $s(n-s)+\binom{s}{2}$ edges. If any color has at least $\binom{s}{2}$ edges missing between $S$ and $V(G_i)\setminus S$, we are done. This implies that at least $n-s-t\binom{s}{2}$ vertices are joined to each vertex of $S$ in each color. If there is a rainbow graph from $F[s]$ inside $S$, then we can find a rainbow $F$. Thus, the collection restricted to $S$ is rainbow $\mathcal{F}[s]$-free. This means that there are at most $ex_t(s,\mathcal{F}[s])$ edges inside $S$. Note that there are at most $s(n-s)$ edges between $S$ and the rest of vertices in each color, completing the proof of \textbf{(ii)}.

For the lower bound of \textbf{(iii)}, consider a $K_{p-1,n-p+1}$ with each edge in each color.
    To prove the upper bound of \textbf{(iii)}, let us consider a largest rainbow matching $M$ and denote its vertex set by $U$. Then $|U|\le 2s$ and each edge of the at least $t-s$ colors that do not appear in $M$ is incident to $U$. We say that an edge is \textit{heavy} if it appears in at least $pq$ colors, 
    where $q=|V(F)|-p$. For a vertex $u$, its \textit{heavy neighbors} are the vertices outside $U$ that are joined to $u$ by heavy edges. Consider a $p$-element subset $U'\subset U$. We claim that there are at most $q-1$ vertices that are common heavy neighbors of each vertex of $U'$. Indeed, otherwise we found a $K_{p,q}$ with $U'$ being one of the parts, and we can pick greedily distinct colors for the edges, so we find a rainbow $K_{p,q}$.

    We obtained that there are at most $(q-1)\binom{2s}{p}$ vertices outside $U$ that have at least $p$ heavy neighbors in $U$, let $H$ denote their set. Let $A_1,\dots, A_{\binom{|U|}{p-1}}$ be the $(p-1)$-element subsets of $U$. We partition $V(G)\setminus (U\cup H)$ to $B_0, B_1,\dots, B_{\binom{|U|}{p-1}}$, where $B_0$ denotes the set of vertices with at most $p-2$ heavy neighbors in $U$ and for $1\le i\le \binom{|U|}{p-1}$, $B_i$ denotes the set of vertices with heavy neighborhood $A_i$.

    %We claim that if $|B_i|\ge ...$, then $u\in A_i$ is incident to at most ... edges in all but ... colors.

    We claim that for each $i$, if $u\not\in A_i$, then in all but $q-1$ colors, at most $q-1$ edges go from $u$ to $B_i$. Indeed, a rainbow star $S_q$ with center $u$ and other vertices in $B_i$ would be extended with $A_i$ to a rainbow $K_{p,q}$, since we can greedily pick the rest of the edges. Thus, we can apply Proposition \ref{starprop}. Let us consider the largest rainbow star with center $u$ and other vertices in $B_i$, then in each color not in the star, any edge between $u$ and $B_i$ goes to the leaves of the star. 
    As there are at most $\binom{2s}{p-1}$ choices for $i$, all but at most $(q-1)\binom{2s}{p-1}$ colors have that for each $i$, at most $(p-1)|B_i|+(2s-p+1)(q-1)$ edges are between $U$ and $B_i$ of that color.

    Let us consider now $v\in B_0$. Clearly, there are at most $(p-2)t+(2s-p+2)(pq-1)$ edges between $U$ and $v$, thus there are at most $|B_0|((p-2)t+(2s-p+2)(pq-1))$ edges between $U$ and $B_0$. Therefore, at most $\frac{(p-2)t+(2s-p+2)(pq-1)}{p-1}$ colors have more than $(p-1)|B_0|$ edges between $U$ and $B_0$. 

    Therefore, if $t>\frac{(p-2)t+(2s-p+2)(pq-1)}{p-1}+s+(q-1)\binom{2s}{p-1}$, then there is a color $j$
    %$t>(2s-p+2)(pq-1)+(p-1)s+(p-1)(q-1)\binom{2s}{p-1}$ is large ... 
    that has all its edges incident to $U$, for each $i$ there are at most $(p-1)|B_i|+(2s-p+1)(q-1)$ edges between $U$ and $B_i$ of color $j$, and there are at most $(p-1)|B_0|$ edges between $U$ and $B_0$ of color $j$. The rest of the edges of color $j$ are inside $U$ or between $U$ and $H$. Therefore there are at most $(p-1)n+O(1)$ edges of color $j$, completing the proof.

    %(iv): for balanced trees satisfying Erd\H os-S\'os: like the uncolored. prof as above+ we change for $B_0$ the neighborhood to the $(p-1)$-set $U'$. Before that, e do that for each $B_i$, because there cannot be edge from $B_i$ to other than $A_i$ because the endpoint would be a leaf in $F$. Oh, and heavy edges we can assume have every color. For $B_0$, this ruins a lot of colors, we show that among the remaining colors, where the number of edges did not increase, there is one with not too many edges. After deleting $B_0$, we do the same for vertices of $U\setminus U'$, one by one, ruining many other colors. In the remaining colors, we have that ... problem: heavyness is ruined when we delete colors. The threshold has to be high to avoid that. But at one place the number of colors ruined depends on $t$. Not a problem, since $pq$-heavy implies $t$-heavy...

The lower bound of \textbf{(iv)} is obtained by taking $K_{p-1,n-p+1}$ in each color, and adding a rainbow $\cF[p-1]$-free collection of graphs in the smaller partite set.

    To prove the upper bound of \textbf{(iv)}, first observe that if an edge $uv$ is $p^2$-heavy (recall that $q=p$ now), then we can assume that $uv$ is $t$-heavy. Indeed, if adding colors to $uv$ would create a rainbow $F$, we could replace the color of $uv$ with a color it originally had such that the copy of $F$ is still rainbow. 
    
    We follow the line of thought of the proof of \textbf{(iii)}. We claim that if $uv$ is an edge in any color, $u\in B_i$ with $i>0$ and $|B_i|\ge p$, then $v\in A_i$. Otherwise, we consider a proper 2-coloring of $F$ with color classes $A$ and $B$ such that part $A$ contains a leaf. We embed the leaf to $v$, its neighbor to $u$, the other vertices of $A$ into $A_i$, and the other vertices of $B$ into $B_i$ arbitrarily. Then all the edges of $F$ are present in the embedding, and all but the single edge incident to $v$ are $p^2$-heavy. Therefore, we can greedily choose colors for those edges to obtain a rainbow $F$, a contradiction.

%We need that there is a $B_j$ with at least $2p$ vertices. Otherwise $B_0$ is large, but there are at most $t(p-2)+2s(p^2-1)$ edges incident to such a vertex, so total number of edges divided by $t$ is small...

We claim that there is a $B_j$ with at least $2p$ vertices. Indeed, otherwise $|B_0|=n-O(1)$. Observe that there are at most $t(p-2)+2s(p^2-1)$ edges that are not in the colors of the matching $M$ incident to a vertex of $B_0$, by the definition of $B_0$. Therefore, the total number of edges incident to $B_0$ in those colors is at most $t(p-2)n+2s(p^2-1)n$. This implies that there is a color with $(p-2)n+O(1)$ edges incident to $B_0$. Clearly there are $O(1)$ edges not incident to $B_0$, completing the proof in this case.

    If $|B_j|\ge 2p$, then we can delete all the edges incident to the vertices of a $B_i$ and join the vertices of $B_i$ to the vertices of $A_j$, in each color, where $i,j>0$. Each vertex of $B_i$ is still incident to $p-1$ edges in each color, and if there is a rainbow $F$ in the new graph, that must contain a vertex of $B_i$, which we can replace by an arbitrary unused vertex of $B_j$. Therefore, we can apply this change and assume that there is at most one non-empty $B_j$. Then we delete the edges incident to $B_0$ and join the vertices of $B_0$ to $A_j$ in each color. Then each color has $(p-1)|B_0|$ edges between $U$ and $B_0$. Recall that in the original collection, at most $\frac{(p-2)t+(2s-p+2)(pq-1)}{p-1}$ colors have more than $(p-1)|B_0|$ edges between $U$ and $B_0$. Therefore, the number of edges did not decrease in the other colors. 
    
    Now we consider a vertex $u\in U\setminus A_j$. If there is no rainbow $S_p$ with center $u$,
    %In this case we can show that there are $p-1$ colors such that deleting them we have at most $p-1$ edges left... 
    %from $u$ with the other endpoint outside $A_j$, 
    then we delete the edges incident to $u$ and join $u$ to $A_j$ in each color. 
    By Proposition \ref{starprop}, the number of edges decreased in at most $p-1$ colors. 

    Let $B_j'$ denote the set of vertices that have neighborhood $A_j$ in all the colors, i.e., the vertices of $B_j$ and the vertices whose neighborhood we changed. Then $|B_j'|=n-O(1)$. The  vertices not in $B_j'\cup A_j$ are the vertices of $H$ and some vertices in $U\setminus A_j$, let $U'$ denote their set. 

    We obtained a new collection where all but $O(1)$ vertices have neighborhood $A_j$ in all the colors. We have $t'\ge t-\frac{(p-2)t+(2s-p+2)(pq-1)}{p-1}-2s(p-1)$ colors such that the number of edges of these colors is not smaller than in the original collection. We call these the \textit{residuary colors}. 

    \begin{clm}
         There is no edge $uv$ with $u\in A_j$ and $v\in U'$ in any of the colors.
    \end{clm}
    \begin{proof}[Proof of Claim]
        We follow the proof of the analogous statement (Claim 2.1) in \cite{gerbner}. It was observed there that $F$ has a vertex $x$ that is adjacent to at least one and at most $p-1$ leaves and to exactly one vertex $y$ of degree greater than one. Let $B$ be the color class of $x$ in the proper coloring and $A$ be the other color class. Then we embed $x$ to $v$ and $y$ to $u$. Recall that $v$ is the center of a rainbow $S_p$, either because there are at least $p$ heavy edges incident to $v$, or because $v\in U$ and we did not change its neighborhood. If $u$ is in the rainbow star, we embed the leaf neighbors of $x$ to the other leaves of the star such that at least one of them is embedded to a vertex $w$ outside $A_j$, and pick the colors of the edges from the rainbow star. If $u$ is not in the star, we pick one of the available colors for the edge $uv$, and at most one edge $vw'$ of the star has this color. Then there is at least one other leaf $w$ in the star outside $A_j$. We embed the other leaves adjacent to $v$ and the incident edges into this star without using $w'$, but using $w$. Afterwards, we embed the rest of $A$ to $A_j$, and the rest of $B$ to $B_j'$ arbitrarily. This is doable, since $A_j$ has $p-1$ vertices, $A$ has $p$ vertices, and at least one of them is embedded to $w\not\in A_j$. This way we found a copy of $F$ where two edges already have distinct colors, and the rest of the edges are $t$-heavy, thus we can greedily pick their color and obtain a rainbow $F$, a contradiction.
    \end{proof}

    Let us return to the proof of the theorem. Observe that $|U'|$ does not depend on $t$. Consider the edges inside $U'$ that are $p^2$-heavy. If there is a copy of $F$ formed by such edges, then we can greedily find a rainbow $F$, a contradiction. Therefore, there are at most $\ex(|U'|,F)\le (p-1)|U'|$ such edges (here we use that the Erd\H os-S\'os conjecture holds for $F$).
    
    Consider the edges inside $U'$ that are not $p^2$-heavy. There are at most $(p^2-1)\binom{|U'|}{2}$ colors on those edges. Therefore, $t''\ge t'-(p^2-1)\binom{|U'|}{2}$ residuary colors do not appear on those edges. We call those colors \textit{good residuary colors}. In those colors,
    there are at most $(p-1)|U'|$ edges inside $U'$. 
    
    We claim that there is no rainbow graph from $\cF[p-1]$ inside $A_j$ in the good residuary colors. Indeed, such a graph could be extended with vertices from $B_j'$ to a copy of $F$, and the new edges are between $A_j$ and $B_j'$, thus $p^2$-heavy. We keep the colors of the edges inside $A_j$ and choose colors greedily for the rest of the edges, thus we obtain a rainbow copy of $F$, a contradiction.

    We obtained that there are at most $\ex_{t''}(p-1,\cF[p-1])$ edges in one of the good residuary colors inside $A_j$. We claim that $\ex_{t''}(p-1,\cF[p-1])=\ex_{t}(p-1,\cF[p-1])$ if $t''$ is sufficiently large. By definition, $\ex_{k}(p-1,\cF[p-1])$ monotone decreases as $k$ increases. Since it can take at most $\binom{p-1}{2}$ values, at one point it becomes constant, in other words there is a $c=c(p,F)$ such that for every $k>c$, $\ex_{k}(p-1,\cF[p-1])=\ex_{c}(p-1,\cF[p-1])$.
    
    %we need to show that there is no rainbow $\cF[p-1]$ inside $A_j$, because we could greedily extend it from $B_j$... but only the residuary colors matter now, and we have some $t''$ colors, so we have to show that it does not matter whether we have $t''$ or $t$...
    
   Therefore, if $t''>c(p,F)$, there is a good residuary color with at most $\ex_{t}(p-1,\cF[p-1])$ edges inside $A_j$, exactly $(p-1)|B_j'|$ edges between $A_j$ and $B_j'$, and at most $(p-1)|U'|$ edges inside $U'$, altogether at most $(p-1)(n-p+1)+\ex_{t}(p-1,\cF[p-1])$ edges, completing the proof.
    %problem: we do not have that $U'$ is $F$-free, only that it is rainbow $F$-free. And clearly the Erdős-Sós does not work there, since we can have quadratic by taking disjoint cliques...
\end{proof}

Theorem \ref{sum} \textbf{(i)} follows from the more general statement below.

\begin{proposition} Let $m=|E(F_1)|-1$. Then
    $\ex_t^\sum(n,\{F_1,\dots,F_k\})\le \ex_m^\sum(n,\{F_2,\dots,F_k\})+(t-m)\ex(n,F_1)$.
\end{proposition}

\begin{proof}
    Let $(G_1,\dots,G_t)$ be a rainbow $\{F_1,\dots,F_k\}$-free collection. We can assume by Lemma \ref{nestlemma} that $G_1\supseteq G_2\supseteq \dots \supseteq G_t$. If $G_{m+1}$ contains $F_1$, then $m+1$ colors contain the same copy of $F_1$, and we can assign distinct colors to each edge of that copy, giving us a rainbow copy of $F_1$, a contradiction. Therefore, while obviously $G_1,\dots,G_m$ do not contain rainbow copies of $F_2, \dots, F_k$, the other colors each contain at most $\ex(n,F_1)$ edges.
\end{proof}

Let us continue with the proof of \textbf{(ii)} of Theorem~\ref{sum} that we restate here for convenience.

\begin{thm*}
  For $n$ sufficiently large, we have $\ex_t^\sum(n,\{K_3,M_{s+1}\})=\ex_s^\sum(n,K_3)$.   
\end{thm*}

\begin{proof} Let $(G_1,\dots,G_t)$ be a rainbow $\{K_3,M_{s+1}\}$-free collection. By Lemma~\ref{nestlemma} we can assume that for each $i\le t-1$, $G_i$ contains $G_{i+1}$ as a subgraph.

    We are done if $G_{s+1}$ is empty. If $G_{s+1}$ is not empty, then we cannot have that $1,\dots,s$ are each strong colors, in particular by the nested property, $s$ is not a strong color. By Lemma~\ref{strong}, there are less than $2sn$ edges in $G_s$, thus also in $G_i$ for $i>s$. Therefore, we have $\sum_{i=s}^t|E(G_i)|\le 2tsn$, while clearly $\sum_{i=1}^{s-1}|E(G_i)|\le (s-1)\binom{n}{2}$. This completes the proof in the case $s\le 2$.

    If $s=3$, then we are done if $G_3$ is empty. Otherwise, for every edge $uv$ of $G_3$ and each vertex $w$ other than $u$,v, we have that either $uw$ or $vw$ is not in $G_2$. In particular, $|E(G_2)|\le\binom{n}{2}-(n-2)$. If there are at most $n-2$ edges in the graphs $G_i$ with $i\ge 3$, we are done. Otherwise, there are at least $(n-2)/t$ edges in $G_3$. If there is an $M_5$ in $G_3$, then $G_4$ must be empty, since any edge of $G_4$ would avoid an $M_3$ from $G_3$. Thus, we only have to avoid a rainbow triangle in $(G_1,G_2,G_3)$ and we are done by the results on $\ex_t^\sum(n,K_3)$ mentioned earlier. 
    Therefore, we can assume that there is no $M_5$ in $G_3$. Then
    we have that the set $S$ of at most 8 vertices of a largest matching in $G_3$ are incident to each edge of $G_3$. Therefore, there is a vertex $v$ that is incident to at least $(n-2)/8t$ edges of $G_3$. Then the at least $\binom{(n-2)/8t}{2}=\Omega(n^2)$ pairs of the neighbors of $v$ do not form edges of $G_1$. Therefore, $|E(G_1)|+|E(G_2)|\le 2\binom{n}{2}-\Omega(n^2)$. Since the other graphs $G_i$ with $i\ge 3$ have $O(n)$ edges, we are done.

    If $s=4$, we will use the idea of Frankl \cite{frankl} from the proof of $\ex_3^\sum(n,K_3)=n(n-1)$. It goes by showing that for each set of three vertices, the sum of the number of colors on the edges between these vertices is at most 6. Adding this up for each triple, every edge is counted exactly $n-2$ times, and we get an upper bound $6\binom{n}{3}$. Let us assume now that $G_3$ contains a vertex $v$ of degree at least $n^{3/4}$. Let $U$ denote the set of neighbors of $v$. Then the edges inside $U$ are not in any $G_i$. When we follow Frankl's proof for the colors 1,2,3, we can see that for the $\binom{n^{3/4}}{3}$ triples inside $U$, the number of colors is 0. Therefore, the total sum is at most $6\binom{n}{3}-6\binom{n^{3/4}}{3}$. Divided by $n-2$, we obtain that $|E(G_1)|+|E(G_2)|+|E(G_3)|\le n(n-1)-cn^{5/4}$ for some constant $c>0$. Since there are $O(n)$ edges in the other graphs, we are done. Let us assume that each vertex has degree less than $n^{3/4}$ in $G_3$ and consider an arbitrary edge $uv$ of $G_3$. For each vertex $w$, we have that either $uw$ or $vw$ is not in $G_2$. We found for every edge of $G_3$ at least $n-2$ missing edges of $G_2$. Each such edge $uw$ is counted at most $2n^{3/4}$ times, since there are at most $2n^{3/4}$ edges of $G_3$ incident to $u$ or $v$. Therefore, at least $(n-2)|E(G_3)|/2n^{3/4}$ edges are missing from $G_2$. In other words, there are at most $n(n-1)-(n-2)|E(G_3)|/2n^{3/4}$ edges in $G_1$ and $G_2$, while there are at most $t|E(G_3)|$ edges in the other  graphs, completing the proof.

    If $s\ge 5$, then we have that $|E(G_1)|+|E(G_2)|+|E(G_3)|+|E(G_4)|\le 4\lfloor n^2/4\rfloor$, thus $|E(G_4)|\le \lfloor n^2/4\rfloor$, therefore $|E(G_i)|\le \lfloor n^2/4\rfloor$ for each $i\ge 4$. These imply the claimed bound \linebreak $\sum_{i=1}^t|E(G_i)|\le t\lfloor n^2/4\rfloor$.

    \end{proof}

Let us continue with the proof of Theorem~\ref{prod} that we restate here for convenience.

\begin{thm*} Let $t\ge\max\{|E(F)|,s+1\}$.

\textbf{(i)} If $n$ is sufficiently large, then $\ex_t^\prod(n,M_{s+1})=(n-1)^{t-s+1}\binom{n}{2}^{s-1}$.

  \textbf{(ii)} If $F$ is not a star with isolated edges added, then $\ex_t^\prod(n,\{F,M_{s+1}\})=\Theta(n^{t+s-1})$.

  \textbf{(iii)}
    \begin{displaymath}
\ex_t^\prod(n,\{S_r,M_{s+1}\})=
\left\{ \begin{array}{l l}
\Theta(n^{2s-2}) & \textrm{if\/ $r=2$},\\
\Theta(n^{s(r-1)-1}) & \textrm{if\/ $t>s(r-1)$ and $r>2$},\\
\Theta(n^{s(r-1)}) & \textrm{if\/ $t=s(r-1)$},\\
\Theta(n^{t+s-\lceil (t-s)/(r-2)\rceil}) & \textrm{otherwise}.\\
\end{array}
\right.
\end{displaymath}

  \textbf{(iv)}  If $F$ is a star $S_r$ with $1\le m\le s-1$ isolated edges added, then 
  \begin{displaymath}
\ex_t^\prod(n,\{F,M_{s+1}\})=
\left\{ \begin{array}{l l}
\Theta(n^{2s-2}) & \textrm{if\/ $r=2$},\\
\Theta(n^{t+s-1}) & \textrm{if\/ $r-1\ge t-s+1$},\\
\Theta(n^{t+m-1}) & \textrm{if\/ $t\ge s(r-1)$},\\
\Theta(n^{t+m-1}) & \textrm{if\/ $r+s-2<t<s(r-1)$ and $m>\frac{s(r-1)-t}{r-2}$}.\\
\Theta(n^{t+\lfloor \frac{s(r-1)-t}{r-2}\rfloor})  & \textrm{if\/ $r+s-2<t<s(r-1)$ and $m\le\frac{s(r-1)-t}{r-2}$}.\\
\end{array}
\right.
\end{displaymath}
\end{thm*}

\begin{proof}
    The lower bound in \textbf{(i)} is given by the complete graph in $s-1$ colors and a star in the rest of the colors. For the upper bound, consider a rainbow $M_{s+1}$-free collection. If there is a color with no edges, then the product is zero and we are done. If there are $s$ strong colors, then they extend any edge in any other color to a rainbow $M_{s+1}$, a contradiction. Thus, there are at most $s-1$ strong colors. By \textbf{(i)} of Lemma~\ref{strong}, there are $O(n)$ edges in the colors that are not strong, thus if there are less than $s-1$ strong colors, we obtain an upper bound $O(n^{t+s-2})$ and we are done. Therefore, we have exactly $s-1$ strong colors. The rest of the colors most form a rainbow $M_2$-free collection, thus by Proposition~\ref{m2}, either there is a $G_i$ with at most 4 edges (giving us an upper bound $O(n^{t+s-2})$), or each color is a star, thus has at most $n-1$ edges. We obtained that $s-1$ colors have at most $\binom{n}{2}$ edges and the rest of the colors have at most $n-1$ edges, completing the proof of \textbf{(i)}.

    The lower bound of \textbf{(ii)} follows by taking a monocolored clique of order $\lfloor n/2s\rfloor$ in $s-1$ colors and a star in all colors on the rest of the vertices. The upper bound follows from \textbf{(i)}.

  Let us continue with \textbf{(iii)}.
Let $\ell=\lfloor n/st\rfloor$ and $k=\lceil (t-s)/(r-2)\rceil$. For the lower bound if $r=2$, we take $s-1$ copies of $K_\ell$ and independent edges in the other colors.
For the lower bound if $t>s(r-1)$ and $r>2$, let us take $s-1$ copies of $S_\ell$, each in $r-1$ distinct colors, and a copy of $S_\ell$ where each edge has $r-2$ colors, and one edge has the remaining $t-s(r-1)+1$ colors. For the lower bound if $t=s(r-1)$, let us take $s$ copies of $S_\ell$, each in $r-1$ distinct colors. For the lower bound if $t<s(r-1)$, let us take $s-k$ monochromatic copies of $K_{\ell+1}$ in distinct colors, $k-1$ copies of $S_\ell$, each in $r-1$ distinct colors, and a copy of $S_\ell$ in the remaining $t-(s-k)-(k-1)(r-1)$ colors.

Let us continue with the upper bound. Let $q$ be a sufficiently large integer.
A theorem of Chv{\'a}tal and Hanson \cite{ch} determines $\ex(n,\{S_q,M_{s+1}\})$ for all values of the parameters. We will only use that $\ex(n,\{S_q,M_{2s+1}\})\le q'=q'(s,q)$. 
We say that a non-strong color is \textit{medium} if it has at least $q'$ edges. By \textbf{(iii)} of Lemma \ref{strong}, such a color $i$ does not contain $M_{2s+1}$. Then by the theorem of Chv{\'a}tal and Hanson, $G_i$ contains $S_q$, i.e., each medium color has a vertex of degree at least $q$. Let us pick such a vertex for each medium color. Assume that there are $m$ medium colors and let $v_1,\dots,v_m$ be the vertices picked. Note that some of them may coincide.  Observe that a vertex has degree at least $q$ in at most $r-1$ colors, otherwise we find a rainbow $S_r$. Therefore, there are at least $\lceil m/(r-1)\rceil$ distinct vertices picked, say $v_1,\dots, v_{\lceil m/(r-1)\rceil}$. Then we claim that we can find a rainbow matching of size $\lceil m/(r-1)\rceil$ using only medium colors. For each medium color, we picked a vertex $v_i$; we say that the color \textit{belongs} to $v_i$. We go through the vertices $v_i$, $i\le \lceil m/(r-1)\rceil$ and each time we pick an edge incident to $v_i$, of a color that belongs to $v_i$. We pick the other endpoint such that it is not among the $v_i$, nor among the other endpoints picked earlier. This is doable since we have at least $q$ choices, and we have to avoid at most $2\lceil m/(r-1)\rceil-2$ vertices. This way we found a matching, and it is rainbow because each color belongs to only one vertex.

%{\color{blue}This implies that there are at most $s-\lceil m/(r-1)\rceil$ strong colors. Each strong color has $O(n^2)$ edges, each medium color has $O(n)$ edges, and other colors have $O(1)$ edges. Therefore, the product is $O(n^{m+2(s-\lceil m/(r-1)\rceil)})$. Observe that $\lceil m/(r-1)\rceil\le s$. Simple algebra implies the upper bound $O(n^{s(r-1)})$ if $r>3$ and the upper bound $O(n^{2s-2})$ if $r=2$. }

This implies that there are at most $s-\lceil m/(r-1)\rceil$ strong colors if $r>2$. Each strong color has $O(n^2)$ edges, each medium color has $O(n)$ edges, and other colors have $O(1)$ edges. Therefore, the product is $O(n^{m+2(s-\lceil m/(r-1)\rceil)})$. Observe that $\lceil m/(r-1)\rceil\le s$. Simple algebra implies the upper bound $O(n^{s(r-1)})$ if $r>2$. If $r=2$, there is no edge that is incident to $v_i$ in each weak color, otherwise, we can find a rainbow $S_2$. This implies that there are at most $s-m-1$ strong colors. Therefore, the product is $O(n^{2s-2})$.

%Let $x=s- m/(r-1)$, then we have $m=(r-1)(s-x)$, and in the brackets we have something less than $x$

Moreover, this is sharp only if we have exactly $s$ distinct vertices picked, and $r-1$ colors belong to each of them. Then if $t>s(r-1)$, we have another color. If an edge of that color is incident to at least one of $v_1,\dots,v_s$, then we find a rainbow $S_r$, otherwise we find a rainbow $M_{s+1}$, a contradiction. Therefore, $m+2(s-\lceil m/(r-1)\rceil)<s(r-1)$, but it is an integer, thus at most $s(r-1)-1$, completing the proof in this case. 

In the case $t<s(r-1)$, let $t'$ denote the number of medium and strong colors. We will use the simple upper bound $m\le (r-1)\lceil m/(r-1)\rceil$. We know that $t'\le\lceil m/(r-1)\rceil(r-1)+s-\lceil m/(r-1)\rceil$, thus simple algebra gives $\lceil m/(r-1)\rceil\ge k$. Then the product is $O(n^{m+2(s-\lceil m/(r-1)\rceil)})=O(n^{t'+s-\lceil m/(r-1)\rceil})=O(n^{t+s-k})$. 

    Let us continue with \textbf{(iv)}. The lower bound if $r=2$ follows from \textbf{(iii)}. If $r-1\ge t-s+1$, the lower bound is given by taking a star $S_{n-1}$ with center $v$ in $t-s+1$ colors, and 
    a monochromatic copy of $K_{\lfloor (n-1)/(s-1)\rfloor}$ in the other $s-1$ colors. The upper bound follows from \textbf{(i)}. 

    For other values of the parameters, the lower bound $\Omega(n^{t+m-1})$ is obtained by taking a star $S_\ell$ with center $v$ in $t-m+1$ colors, and 
    a monochromatic copy of $K_\ell$ in the other $m-1$ colors. The other lower bound is obtained by taking $K_\ell$ in $\lfloor \frac{s(r-1)-t}{r-2}\rfloor$ colors, and an $\{S_r,M_{s+1-\lfloor \frac{s(r-1)-t}{r-2}\rfloor}\}$-free collection in the other colors, given by the constructions from \textbf{(iii)}.
    %follows from \textbf{(iii)}.

       For the upper bound, we need to modify the definition of strong color a little bit. We say that a color $i$ is \textit{very strong} if for any rainbow copy of $S_r$ plus at most $m-1$ isolated edges, there is an edge of color $i$ that does not contain any vertex of that copy of $S_r$, nor any vertex of the isolated edges. We say that a color is \textit{weak} otherwise. One can show similarly to Lemma \ref{strong} that weak colors have $O(n)$ edges. If we have $p<m$ very strong colors, then we have $p$ colors with $O(n^2)$ edges and $t-p$ colors with $O(n)$ edges, giving us the upper bound $O(n^p)=O(n^{t+m-1})$.

       Assume now that we have $p\ge m$ very strong colors. Then the other $t-p$ colors must avoid rainbow $M_{s+1-p}$ and rainbow $S_r$. Then we can apply \textbf{(iii)}, and have the upper bound

    \begin{displaymath}
\ex_t^\prod(n,\{F,M_{s+1}\})\le n^{2p}\ex_{t-p}^\prod(n,\{S_r,M_{s+1-p}\})=
\left\{ \begin{array}{l l}
\Theta(n^{2s-2}) & \textrm{if\/ $r=2$},\\
\Theta(n^{(s-p)(r-1)-1+2p}) & \textrm{if\/ $t-p>(s-p)(r-1)$},\\
\Theta(n^{(s-p)(r-1)+2p}) & \textrm{if\/ $t-p=(s-p)(r-1)$},\\
\Theta(n^{t+s-\lceil (t-s)/(r-2)\rceil}) & \textrm{otherwise}.\\
\end{array}
\right.
\end{displaymath}

If $r=2$, then we are done.
Observe that if $r>2$, then an upper bound of the form $O(n^{s(r-1)-1-p(r-3)})$ is maximized if $p$ is minimized, i.e., $p=m$. Simple algebra completes the proof in this case. %since $s(r-1)<(r-2)m+t$, this bound is less than $t+m-1$.
If $m>\frac{s(r-1)-t}{r-2}$, then $p\ge m$ implies that $t-p>(s-p)(r-1)$. Therefore, in this case we have the upper bound $O(n^{s(r-1)-1-p(r-3)})$.

If $m\le\frac{s(r-1)-t}{r-2}$, then $p\ge m$ implies that $t-p\le (s-p)(r-1)$. If $\frac{s(r-1)-t}{r-2}$ is an integer, then this may be equal to $p$, which would give the upper bound $O(n^{t+\frac{s(r-1)-t}{r-2}})$ and we are done. Therefore, no matter whether $\frac{s(r-1)-t}{r-2}$ is an integer, we are left with the other bounds $O(n^{(s-p)(r-1)-1+2p})$ and $O(n^{t+s-\lceil (t-s)/(r-2)\rceil})$. We have already dealt with the first case. Observe that $(t-s)/(r-2)=s-\frac{s(r-1)-t}{r-2}$. Therefore, the second bound is equal to $t+\lfloor\frac{s(r-1)-t}{r-2}\rfloor$.
\end{proof}

\section{Concluding remarks}
We have proved that $\ex_t^\sum(n,\{F,M_{s+1}\})=\ex_s^\sum(n,F)+O(n)$, and in the case $F=K_3$, we managed to get rid of the linear additive error term. 
We note that $\ex_t^\sum(n,\{F,M_{s+1}\})=\ex_s^\sum(n,F)$ trivially holds for many $t,n,F,s$. In particular, this is the case if $\ex_t^\sum(n,F)=(|E(F)|-1)\binom{n}{2}$ and $n$ is sufficiently large. Since we have $\ex_t^\sum(n,M_{s+1})=s\binom{n}{2}$ by Proposition \ref{proppi} below, we have that $\ex_t^\sum(n,\{F,M_{s+1}\})\le q\binom{n}{2}$, where $q=\min\{|E(F)|-1,s\}$. Since $K_n$ in $q$ colors do not contain rainbow $F$, nor rainbow $M_{s+1}$, indeed we have $\ex_t^\sum(n,\{F,M_{s+1}\})=\ex_s^\sum(n,F)$. 
%\begin{proposition}\label{propp}
%    $\ex_t^\sum(n,M_{s+1})=s\binom{n}{2}$.
%\end{proposition}
%\begin{proof}
%Let $(G_1,\dots,G_t)$ be a rainbow $F$-free collection with $n$ vertices. By Lemma \ref{nestlemma}, we assume $G_1\supseteq G_2\supseteq \dots \supseteq G_t$. If $G_{s+1}$ is empty, we are done. We assume that $G_{s+1}$ is nonempty. Then $s$ is not strong color, otherwise, we can find a rainbow $M_{s+1}$ in collection $(G_1,\dots,G_{s+1})$. By (i) of Lemma \ref{strong}, $e(G_s)\le 2sn$. According to the nested property, we have $\sum^t_{i=1}e(G_i)\le (s-1)\binom{n}{2}+2tsn$, completing the proof.
%\end{proof}

\begin{proposition}\label{proppi}
    Let $F$ be a bipartite graph and $t\ge |E(F)|$, then for sufficiently large $n$ we have $\ex_t^\sum(n,F)=(|E(F)|-1)\binom{n}{2}$.
\end{proposition}

Note that \cite{kssv} focused on the case $t$ increases with $n$, which may be a reason why they did not state this simple result.

\begin{proof}
    Let $(G_1,\dots,G_t)$ be a rainbow $F$-free collection with vertex set $V$ and the largest sum of edges. We can assume by Lemma \ref{nestlemma} that $G_1\supseteq G_2\supseteq \dots \supseteq G_t$. Let $f:=|E(F)|-1$. Then for $i>f$, $G_i$ is $F$-free, thus $|E(G_i)|=o(n^2)$ by the K\H ov\'ari-S\'os-Tur\'an theorem \cite{KST}. Therefore, there are $o(n^2)$ pairs $u,v$ such that $uv\not \in E(G_f)$, we say that $uv$ is a non-edge. Let $A$ denote the set of non-edges and $B$ denote the set of edges in $G_f$ that are not contained in any copy of $F$ in $G_f$. Observe that for every $i>f$, every edge of $G_i$ must belong to $B$.

    Let $uv\in B$. There are $\Theta(n^{f-1})$ copies of $F$ containing $uv$ in the complete graph on $V$, and each of those copies contains a non-edge. A non-edge that contain neither $u$, nor $v$ is in $O(n^{f-3})$ copies of $F$ together with $uv$ in the complete graph, thus $o(n^{f-1})$ copies of $F$ in the complete graph contain at least one of the $o(n^2)$ such non-edges. A non-edge that contains $u$ or $v$ is in $O(n^{f-2})$ copies of $F$ together with $uv$ in the complete graph, thus we must have $\Theta(n)$ non-edges containing $u$ or $v$. 

    If a vertex $w$ is incident to $\Theta(n)$ edges in $B$, then let $W$ consist of the other endpoints of these edges. Let $F'$ denote a graph obtained from $F$ by deleting a vertex, then $G_f[W]$ is $F'$-free, thus has $o(|W|^2)$ edges. But this means that there are $\Theta(|W|^2)=\Theta(n^2)$ non-edges, a contradiction. We obtained that each vertex is incident to $o(n)$ edges in $B$. For each edge of $A$, we have found $\Theta(n)$ edges of $B$, and each edge of $B$ is counted $o(n)$ times, thus $|B|=o(|A|)$. Since there are at most $t|B|$ edges altogether in the graphs $G_i$ with $i>f$, we have at most $f\binom{n}{2}-|A|+t|B|<f\binom{n}{2}$ edges in the collection, completing the proof.
\end{proof}

\bigskip
\textbf{Funding}: 

The research of Gerbner is supported by the National Research, Development and Innovation Office - NKFIH under the grant KKP-133819.

The research of Miao is supported by the China Scholarship Council (No. 202406770056).

\end{document}